\documentclass[11pt]{article}
\usepackage{amsthm, amsmath, amssymb, amsfonts, url, booktabs, tikz, setspace, fancyhdr, bm}
\usepackage[hidelinks]{hyperref}
\usepackage{geometry}
\usepackage{hyperref, enumerate}
\usepackage[shortlabels]{enumitem}
\usepackage[babel]{microtype}
\usepackage[english]{babel}
\usepackage[capitalise]{cleveref}
\usepackage{comment}
\usepackage{bbm}
\usepackage{csquotes}
\usepackage{mathrsfs}
\usepackage{mathabx}

\counterwithin{figure}{section}

\newtheorem{theorem}{Theorem}[section]

\newtheorem{lemma}[theorem]{Lemma}

\newtheorem{claim}[theorem]{Claim}

\theoremstyle{definition}
\newtheorem{definition}[theorem]{Definition}

\theoremstyle{remark}
\newtheorem*{remark}{Remark}

\newenvironment{poc}{\begin{proof}[Proof of claim]}{\end{proof}}

\DeclareMathOperator{\vol}{vol}

\usepackage{todonotes}

\newcommand{\E}{\mathbb{E}}

\newcommand{\I}{\mathbb{I}}
\newcommand{\N}{\mathbb{N}}
\newcommand{\PP}{\mathbb{P}}
\newcommand{\R}{\mathbb{R}}
\newcommand{\wt}{\mathrm{wt}}
\newcommand{\x}{\bm{x}}
\newcommand{\y}{\bm{y}}
\newcommand{\z}{\bm{z}}
\newcommand{\cB}{B}
\newcommand{\cC}{\mathcal{C}}
\newcommand{\cL}{\mathcal{L}}
\newcommand{\cI}{\mathcal{I}}
\newcommand{\cE}{\mathcal{E}}

\newcommand{\eps}{\varepsilon}

\def\Var{\mathop{\rm Var}}

\tikzstyle{p}+=[fill=black, circle, minimum width = 1pt, inner sep =
1pt]
\tikzstyle{w}+=[fill=white, draw, circle, minimum width = 1pt, inner sep =
1.5pt]

\author{	  
Jaehoon Kim\thanks{Department of Mathematical Sciences, KAIST, South Korea. E-mail: {\tt jaehoon.kim@kaist.ac.kr}. Supported by the POSCO Science Fellowship of POSCO TJ Park Foundation, and by the KAIX Challenge program of KAIST Advanced Institute for Science-X.}
\and
	Hong Liu\thanks{Mathematics Institute, University of Warwick, UK. E-mail: {\tt h.liu.9@warwick.ac.uk}. Supported by the UK Research and Innovation Future Leaders Fellowship MR/S016325/1.}
	\and
	Tuan Tran\thanks{
		Discrete Mathematics Group, Institute for Basic Science (IBS), South Korea.
		E-mail: {\tt tuantran@ibs.re.kr}. Supported by the Institute for Basic Science (IBS-R029-Y1).}  
 }

\begin{document}

\date{}

\title{Exponential decay of intersection volume with applications on list--decodability and Gilbert--Varshamov type bound}
\maketitle

\begin{abstract}
	We give some natural sufficient conditions for balls in a metric space to have small intersection. Roughly speaking, this happens when the metric space is (i) expanding and (ii) well-spread, and (iii) a certain random variable on the boundary of a ball has a small tail. As applications, we show that the volume of intersection of balls in Hamming, Johnson spaces and symmetric groups decay exponentially as their centers drift apart. To verify condition (iii), we prove some large deviation inequalities `on a slice' for functions with Lipschitz conditions.
	
	We then use these estimates on intersection volumes to
	\begin{itemize}
	    \item obtain a sharp lower bound on list-decodability of random $q$-ary codes, confirming a conjecture of Li and Wootters; and
	    
	    \item improve the classical bound of Levenshtein from 1971 on constant weight codes by a factor linear in dimension, resolving a problem raised by Jiang and Vardy.
	\end{itemize}
	
	Our probabilistic point of view also offers a unified framework to obtain improvements on other Gilbert--Varshamov type bounds, giving conceptually simple and calculation-free proofs for $q$-ary codes, permutation codes, and spherical codes. Another consequence is a counting result on the number of codes, showing ampleness of large codes. 
\end{abstract}


\section{Introduction}

A well-known fact in convex geometry states that the volume of the intersection of two Euclidean balls of the same radius in $\R^n$ is exponentially (in $n$) smaller than the two given balls. 
It can be proved by observing that the intersection is contained in a ball of smaller radius centered at the mid-point of the centers of the two original balls. This simple proof, however, does not extend to some discrete settings, as the intersection might no longer be enclosed by a ball of smaller radius. One such example is that of the Hamming space over a finite alphabet, one of the most studied space in theoretical computer science and information theory. Indeed, take the 
discrete cube $\{0,1\}^n$ endowed with the Hamming metric and let $k,r\in\N$ with $2k\le r$. Consider the two radius-$r$ Hamming balls $A$ and $B$ centered at $a=0^n$ and $b=1^{2k}0^{n-2k}$ respectively. Take a mid-point $c$ of $a$ and $b$, say by symmetry $c=1^k0^{n-k}$. Then the point $x=0^k1^r0^{n-r-k}$ lies in the intersection $A\cap B$, but it is of Hamming distance $r+k$ from the chosen mid-point $c$.

The expression of the intersection volume in such discrete metric spaces can usually be written out explicitly. The problem is that such expression is often cumbersome and it is a grueling task to estimate. To illustrate, let us consider the $q$-ary Hamming space $\{0,1,\ldots,q-1\}^n$. Denote by $\vol_q(n,r)$ the volume of a radius-$r$ $q$-ary Hamming ball, and by $\vol_q(n,r;k)$ the volume of the intersections of two radius-$r$ balls whose centers are distance $k$ apart. It is not hard to show that
the intersection volume is
\begin{equation}\label{eq:int-q}
\vol_q(n,r;k)=\sum_{i+j\le k}\frac{k!}{i!j!(k-i-j)!}(q-2)^{k-i-j}\sum_{t\le t_{\max}}\binom{n-k}{t}(q-1)^t,
\end{equation}
where $t_{\max}:=\min(n-k,r-k+i,r-k+j)$. 

Estimating the asymptotics of the right hand side above is not at all a straightforward task. Indeed, when $k$ and $r$ are linear in the dimension $n$, Jiang and Vardy~\cite{JV04} studied the binary case $q=2$ with the help of computer. Later, Vu and Wu~\cite{VW05} estimated the general $q$-ary case for all $q\ge 2$ using a discrete analog of Lagrange’s multiplier and some inequalities on entropy functions; their proof, though computer-free and much cleaner, is still rather involved.

\smallskip

Consider the following alternative probabilistic approach to estimate the intersection volume. Let $A,B$ be two radius-$r$ Euclidean balls centered at $a,b\in\R^n$ respectively. Let $\bm{x}$ be a uniform random point drawn from $A$, then the ratio of the volume of the intersection $A\cap B$ and the volume of the radius-$r$ ball is precisely the probability that $\bm{x}$ lies in $A\cap B$, that is, $\frac{\vol(A\cap B)}{\vol(A)}=\PP(\x\in A\cap B)$. We can then bound the probability $\PP(\x\in A\cap B)$ using for instance Talagrand's celebrated deviation inequality~\cite{Tal} for functions with Lipschitz condition with respect to both $\ell_1^n$ and $\ell_2^n$ norms. We refer the readers to~\cite{BL, Mau}
for related results on concentration of measure.

We use this probabilistic approach to give some natural sufficient conditions that guarantee small intersection of balls in a metric space. The advantage of this approach is that it can be implemented in the discrete settings, provided that appropriate concentration inequalities can be proved. 


\subsection{Sufficient conditions for small intersection}

To state our result, we need some notations. Let $(X,\mathsf{d})$ be a finite metric space with $\mathsf{d}$ taking values in $\N\cup\{0\}$.
For $a\in X$ and $r\in \N$, we write $B(a,r)$ for the ball of radius $r$ around $a$ and write $S(a,r)$ for the shell of all points of distance exactly $r$ from $a$. We say the metric space $(X,\mathsf{d})$ has \emph{exponential growth at radius $r$ with rate $c$} if for every $a\in X$ and every $t<r$,
$$\frac{\vol(B(a,r-t))}{\vol(B(a,r))} \le 2e^{-ct}.$$
For $a,b\in X$, let $\ell_{a,b}:~X\rightarrow \R$ be given by
\begin{equation}\label{defn:f-slice}
\ell_{a,b}(x)=\mathsf{d}(x,b)-\mathsf{d}(x,a).
\end{equation}
Given $r,k\in\N$ and $\alpha>0$, we say that the metric space $(X,\mathsf{d})$ is \emph{$(r,k)$-dispersed} with constant $\alpha$ if for any $a,b\in X$ with 
$\mathsf{d}(a,b)=k$ 
and any $0\le i\le \alpha k$, 
$$\E_{\bm{x}\sim S(a,r-i)}\Big[\ell_{a,b}(\bm{x})\Big]\ge 2\alpha k,$$
where $\bm{x}$ is a uniform random point of $S(a,r-i)$. 

A real-valued random variable $\bm{X}$ is \emph{$K$-subgaussian} if for any $t\ge 0$,\[
    \PP(|\bm{X}|\ge t) \le 2\exp\left(-t^2/K\right).
    \]
    
Our result reads as follows. 

\begin{theorem}\label{thm:suff-cond-small-int}
  Let $(X,\mathsf{d})$ be a finite metric space with $\mathsf{d}$ taking values in $\N\cup\{0\}$ and let $k,r\in\N$.  
  Suppose  
  \begin{itemize}
      \item[\rm (A1)] $(X,\mathsf{d})$ has exponential growth at radius $r$ with 
      rate $c>0$; 
      
      \item[\rm (A2)] $(X,\mathsf{d})$ is $(r,k)$-dispersed with constant $\alpha>0$;
   
      \item[\rm (A3)] For any $a,b\in X$ with $\mathsf{d}(a,b)=k$ and any $0\le i\le \alpha k$, $\ell_{a,b}(\bm{x})-\E \ell_{a,b}(\bm{x})$ is $K$-subgaussian, where $\ell_{a,b}$ is as in~\eqref{defn:f-slice} and $\bm{x}$ is drawn uniformly from $S(a,r-i)$.
  \end{itemize}
  Then, for any $a,b\in X$ with $\mathsf{d}(a,b)=k$, 
  $$\frac{\vol(B(a,r)\cap B(b,r))}{\vol(B(a,r))}=2e^{-\Omega_{c,\alpha}(1)\cdot (k+k^2/K)}.$$
\end{theorem}

The merit of \cref{thm:suff-cond-small-int} is its formulation. The conditions (A1)--(A3) are quite natural; they are inspired by properties of the Euclidean space.
By our result, showing that the intersection volume is small then amounts to verifying these conditions, which are more manageable. For instance, using~\cref{thm:suff-cond-small-int}, we can get a conceptually simple and \emph{calculation-free} proof that the intersection volume in~\eqref{eq:int-q} is exponentially small, i.e.~$\frac{\vol_q(n,pn;k)}{\vol_q(n,pn)} \le e^{-\Omega(k)}$, for the optimal range $0<p<1-1/q$ and all $1\le k \le n$ (\cref{lem:volume cube}). It is important that the exponential bound holds for not just when $k=\Omega(n)$, but for all $k$, which is needed in some applications, e.g. the tightness on list-decoding capacity theorem (\cref{thm:list size}).

To illustrate the power of \cref{thm:suff-cond-small-int}, apart from the Hamming cube example above, we shall apply it to Johnson space (\cref{lem:volume constant-weight}) and permutation group (\cref{lem:volume permutation}). 
Such estimates on the intersection volume of balls are useful for various problems. We will use them in~\cref{sec:application} to obtain results on list-decodability of $q$-ary random codes with rate just below the limiting rate, and improvements on Gilbert--Varshamov type bounds for constant weight codes, $q$-ary codes, permutation codes and spherical codes, and the corresponding counting results.

\smallskip

In order to apply~\cref{thm:suff-cond-small-int}, it is not hard to check that the discrete metric spaces we consider have the exponential growth and they are well-dispersed. To verify the third condition that the centered random variable $\ell_{a,b}(\bm{x})-\E \ell_{a,b}(\bm{x})$ is subgaussian in our applications, we prove some concentration inequalities for Lipschitz functions defined on `slices' of the space, see~\cref{lem:slice-concentration-q-ary,lem:slice-concentration-perm}.

\medskip

\noindent\textbf{Notations.} Before discussing the applications in details, let us review the terminology that will be used throughout the paper.
A \emph{code} over a finite alphabet $\Sigma$ is simply a subset of $\Sigma^n$; the number $n$ is referred to as the \emph{length} of the code. The elements of the code are called \emph{codewords}. If $|\Sigma|=q$, the code is called $q$-ary code, with the term binary used for the case $q=2$. We say that the code has \emph{rate} $R$ if the number of codewords is $|\Sigma|^{Rn}$.  
Given two words $x=(x_1,\ldots,x_n)$ and $y=(y_1,\ldots,y_n)$ in $\Sigma^n$, the \emph{Hamming distance} $\Delta(x,y)$ between $x$ and $y$ is the number of coordinates $i$ in which $x_i$ and $y_i$ differ.
For a word $x$ we denote by $x_i$ the value of its $i$-th coordinate. For $x\in\{0,1,\ldots,q-1\}^n$, we denote its \emph{weight}, which is the number of non-zero entries in $x$, by $\wt(x)$. The \emph{Johnson distance} between two binary words $x,y\in \{0,1\}^n$ of the same weight is half of their Hamming distance.
The $q$-ary entropy function $h_q\colon [0,1]\rightarrow \R$ is
\[
h_q(x)=x\log_q(q-1)-x\log_qx-(1-x)\log_q(1-x).
\]
We generally use boldface letters for random variables. Given a finite set $A$, we write $\x\sim A$ for a discrete random variable $\x$ chosen uniformly from $A$.


\section{Applications}\label{sec:application}

\subsection{List decoding of random codes}
One of the main goals of the theory of error-correcting codes is to understand the trade-off between the rate of a code and the fraction of errors the code can tolerate during transmission over a noisy channel. There are two natural error models: Hamming’s adversarial noise model, and Shannon’s stochastic noise model. Channels in Shannon's world can flip each transmitted bit with certain probability, independently of other bits, while channels in Hamming's world can corrupt the codeword arbitrarily, subject only to a bound on the total number of errors. 

There is a gap between Hamming and Shannon's world: one can correct twice as many errors in Shannon's world. We refer the reader to \cite{GRS19} for a thorough comparison. List decoding, which was introduced by Elias \cite{Eli57} and Wonzencraft \cite{Woz58} in the late 1950's, can be used to bridge the gap.  In list decoding we give up unique decoding, allowing decoder to output a list of all codewords that are within Hamming distance $pn$ from the received word. Thus, if at most $pn$ errors occur, the list will include the correct codeword. 
Formally, we say that a $q$-ary code $\cC \subset \Sigma^n$ is \emph{$(p,L)$-list decodable} if any Hamming ball of radius $pn$ in $\Sigma^n$ contains at most $L$ codewords.

List decoding has three important parameters: the rate $R$ of the code, the error fraction $p$, and the list size $L$. A fundamental question in list decoding is to determine the feasible region of $(R,p,L)$.  Despite significant efforts, a full description remains elusive. In 1981, Zyablov and Pinsker \cite{ZP81} proved the list-decoding capacity theorem, thus giving a partial solution to the above question.

\begin{theorem}[Zyablov and Pinsker]
\label{thm:capacity}
Let $q\ge 2, 0<p<\frac{q-1}{q}$, and $\eps>0$.
\begin{itemize}
    \item[1.] There exist $q$-ary codes of rate $1-h_q(p)-\eps$ that are
$\left(p,\lceil\frac{1}{\eps}\rceil\right)$-list decodable.
    \item[2.] Any $q$-ary code of rate $1-h_q(p)+\eps$ that is $(p,L)$-list decodable must have $L\ge q^{\Omega(\eps n)}$.
\end{itemize}
\end{theorem}

\cref{thm:capacity} establishes the optimal trade-off between the rate and the error fraction for list decoding. In particular, it shows that the list decoding capacity is $1-h_q(p)$, which matches the capacity of Shannon's model.

The existential part of \cref{thm:capacity} was achieved by demonstrating that a random code of rate $1-h_q(p)-\eps$ is $(p,\lceil\frac{1}{\eps}\rceil)$-list decodable with high probability. Rudra \cite{Rud11} proved that this result is best possible up to a constant factor, in the sense that a random code of rate $1-h_q(p)-\eps$ requires $L=\Omega_{p,q}(1/\eps)$. In \cite{GN14}, Guruswami and Narayanan provided a more direct proof of Rudra's result.
For binary codes, Li and Wootters \cite{LW21} recently sharpened the
argument of Guruswami and Narayanan to show that the list size of $1/\eps$ in \cref{thm:capacity} is tight even in the leading constant factor:

\begin{theorem}[Li and Wootters]
\label{thm:binary list-size}
For any $p\in (0,1/2)$ and $\eps>0$, there exist $\gamma
_{p,\eps}=\exp\left(-\Omega_{p}\left(\frac{1}{\eps}\right)\right)$ and $n_{p,\eps}\in \N$ such that for all $n\ge n_{p,\eps}$, a random code $\bm{\cC}\subseteq \{0,1\}^n$ of rate $R=1-h(p)-\eps$ is with probability $1-\exp(-\Omega_{p,\eps}(n))$ not $(p,\frac{1-\gamma_{p,\eps}}{\eps}-1)$-list decodable. \end{theorem}

Li and Wootters \cite{LW21} conjectured that \cref{thm:binary list-size} generalizes to $q$-ary codes, for any $q\ge 3$. To quote their words, \textquote{our arguments only work for binary codes and do not extend to larger alphabets.} 

Our first application, making use of the intersection volume estimate, confirms their conjecture, showing that the list size $1/\eps$ in list decoding capacity theorem is optimal for all $q\ge 2$.

\begin{theorem}\label{thm:list size}
Let $q\ge 2$. Then for any $p\in (0,\frac{q-1}{q})$ and $\eps>0$, there exist $\gamma
_{p,q,\eps}=\exp\left(-\Omega_{p,q}\left(\frac{1}{\eps}\right)\right)$ and $n_{p,q,\eps}\in \N$ such that for all $n\ge n_{p,q,\eps}$, a random code $\cC\subseteq [q]^n$ of rate $R=1-h_q(p)-\eps$ is with probability $1-\exp(-\Omega_{p,q,\eps}(n))$ not $(p,\frac{1-\gamma_{p,q,\eps}}{\eps}-1)$-list decodable.
\end{theorem}


\subsection{The sphere-covering bounds}

Our second group of applications concern codes over metric spaces. 
Consider a metric space $(X,\mathsf{d})$ and a real number $r>0$. We say a subset $C$ of $X$ is an \emph{$(X,\mathsf{d},r)$-code} if
$\mathsf{d}(c,c')>r$ for any distinct codewords $c,c'\in C$. A simple covering argument shows the existence of such a code $C$ with
\begin{equation}\label{eq:covering bound}
|C|\ge \inf_{a\in X}\frac{m(X)}{m(B(a,r))}   
\end{equation}
for {\em any} finite measure $m$ on the Borel $\sigma$-algebra of $X$. 
To see why \eqref{eq:covering bound} holds, one can assume $C$ is a maximal $(X,\mathsf{d},r)$-code of finite size. From the maximality of $C$, we deduce that $X=\bigcup_{a\in C}B(a,r)$. By the subadditivity of measures, we then get $m(X) \le \sum_{a\in C}m(B(a,r)) \le |C|\cdot \sup_{a\in X} m(B(a,r))$, resulting in \eqref{eq:covering bound}.

Improving upon the sphere-covering bound \eqref{eq:covering bound} is a notoriously difficult problem; more on this later.
Our next result improves the bound, assuming some mild conditions on the metric space.

\begin{theorem}\label{thm: existence of code}
  Let $(X,\mathsf{d})$ be a finite metric space, and let $r>0$.  Suppose
  \begin{itemize}
      \item[\rm (P1)] (Homogeneous) For every $s\in \mathbb{R}$, all the balls of radius $s$ have the same volume $\vol(s)$.
        \end{itemize}
  Suppose further that there exist $t\in (0,r)$ and $K>0$ such that
  \begin{itemize}
      \item[\rm (P2)] (Exponential growth)  $\frac{\vol(r-t)}{\vol(r)}\le e^{-K}$; and
      
      \item[\rm (P3)] (Small intersection volume) for any $a,b \in X$ with $r-t < \mathsf{d}(a,b) \le r$, 
      $\frac{\vol\left(B(a,r)\cap B(b,r)\right)}{\vol(r)}\le e^{-K}$.
  \end{itemize}
  Then there is an $(X,\mathsf{d},r)$-code of size
  $(1-o_{K\rightarrow \infty}(1))K \cdot \frac{|X|}{\vol(r)}$, and the number of $(X,\mathsf{d},r)$-codes is at least $\exp \left((\frac{1}{8}+o_{K\rightarrow\infty}(1))K^2\cdot \frac{|X|}{\vol(r)}\right)$.
\end{theorem}

\cref{thm: existence of code} can be used in conjunction with \cref{thm:suff-cond-small-int} (for verifying condition (P3)) to give a unified proof of improvements on Gilbert-Varshamov type bounds on various models of error correction codes, which we now discuss in details. Theorem 2.4 builds on recent developments on some graph theoretic results; such approach was pioneered by the work of Jiang and Vardy \cite{JV04} and by Krivelevich, Litsyn and Vardy \cite{KLV04}.


\subsubsection*{$q$-ary codes}
A $q$-ary code $\cC$ is said to have minimum distance at least $d$ if any two codewords in $\cC$ have distance at least $d$. Given three parameters $q,n$ and $d$, what is the largest possible size $A_q(n,d)$ of a $q$-ary length-$n$ code with minimum distance at least $d$? This question has been studied extensively for almost seven decades, and remains one of the most important questions in coding theory. 

For a word $x\in [q]^n$, the Hamming ball of radius $d$ centered at $x$ is the collection of words in $[q]^n$ with distance at most $d$ from $x$. The volume of this ball does not depend on the location of $x$ and can be expressed as
\[
\vol_q(n,d)=\sum_{i=0}^d\binom{n}{i}(q-1)^i.
\]
The sphere-covering bound \eqref{eq:covering bound}, applied to the the Hamming space $([q]^n,\Delta)$, gives
\[
A_q(n,d+1)\ge \frac{q^n}{\vol_q(n,d)}.
\]
This is known in the literature as the famous Gilbert--Varshamov bound \cite{Gil52,Var57} from the 1950's. For five decades this was the best asymptotic lower bound for $A_q(n,d+1)$ (see for example \cite[page 95]{HP10}).

The case when $d$ is proportional to $n$, that is, $d/n$ is a positive constant, is of special interest in coding theory. It is an easy exercise to see that for $d/n \ge (q-1)/q$, the fraction $q^n/\vol_q(n,d)$ is less than $2$. In this case, the Gilbert--Varshamov bound gives no useful information. Thus, the value $(q-1)/q$ is a natural threshold for the ratio $d/n$.

In a breakthrough, Jiang and Vardy \cite{JV04} improved the Gilbert--Varshamov bound, for the binary case, for $d \le 0.4994n$.  
Extending the work of Jiang and Vardy, Vu and Wu \cite{VW05} proved that if $d/n$ is less than $(q-1)/q$, then one can improve the Gilbert--Varshamov bound by a factor linear in $n$.
We give a short proof of the following strengthening of Vu-Wu's result, showing ampleness of large codes.

\begin{theorem}\label{thm:q-ary}
  Let $q\ge 2$ and let $0<p<\frac{q-1}{q}$ and $d=pn$. Then there exists a positive constant $c=c_{p,q}$ such that the number of $q$-ary length-$n$ codes with minimum distance at least $d+1$ is at least
  \[
  \exp\left(cn^2\cdot \frac{q^n}{\vol_q(n,d)}\right).
  \]
\end{theorem}
As the number of subsets of $[q]^n$ of size $o_{p,q}(1)n\cdot \frac{q^n}{\vol_q(n,d)}$ is 
$\exp\left(o_{p,q}(1)n^2\cdot \frac{q^n}{\vol_q(n,d)}\right)$, \cref{thm:q-ary} recovers the bound $A_q(n,d+1) \ge \Omega_{p,q}(1)n\cdot \frac{q^n}{\vol_q(n,d)}$ of Vu and Wu. The original proof of Vu-Wu's bound was quite complicated, and involved heavy calculations. Our proof of \cref{thm:q-ary} is conceptual and reflects, in a clean way, the necessity of the assumption $d/n<(q-1)/q$. 


\subsubsection*{Constant-weight codes}
Given positive integers $n,d$ and $w$, we denote by $A(n,d,w)$ the size of a largest {\em constant-weight code} of length $n$ and minimum Johnson distance $d$ all
of whose codewords are in $\{0,1\}^n$ with weight $w$. Estimating $A(n,d,w)$ accurately is the central problem regarding constant-weight codes. With the exceptions of a few particular small cases \cite{Brouwer} and the fixed $w$ case  \cite{CL07}, it remains open in general.

Thanks to symmetry, all Johnson ball of radius $d$ in $\binom{[n]}{w}$ have the same volume
\[
V_w(n,d):=\sum_{i=0}^{d}\binom{w}{i}\binom{n-w}{i}.
\]
Thus, the sphere-covering bound, specialized to the Johnson space, gives
\[
A(n,d+1,w)\ge \frac{\binom{n}{w}}{V_w(n,d)}.
\]
This lower bound was obtained by Levenshtein back in 1971 \cite{Lev71}.

Our next result provides an improvement on this 50-year-old bound of Levenshtein by a factor linear in the dimension. This resolves a problem posed by Jiang and Vardy \cite{JV04}.

\begin{theorem}\label{thm:constant-weight}
Let $\alpha$ and $\lambda$ be constants satisfying $0<\alpha<\lambda(1-\lambda)$. There is a positive constant $c=c_{\alpha,\lambda}$ such that for $d=\alpha n$ and $w=\lambda n$
\begin{equation*}
A(n,d+1,w)\ge cn \cdot \frac{\binom{n}{w}}{V_w(n,d)}.    
\end{equation*}
\end{theorem}


\subsubsection*{Permutation codes}
Let $S_n$ be the symmetric group of permutations on $[n]$.  Consider a permutation $\sigma \in S_n$ as a codeword $(\sigma(1),\dots, \sigma(n)) \in [n]^n$, then $S_n$ is a subset of $[n]^n$.
With this view, the Hamming distance between two permutations $\sigma, \tau\in S_n$ 
is naturally defined as 
\[ \Delta(\sigma,\tau) = \big|\{i\in [n]: \sigma(i)\neq \tau(i)\}\big|. \]
A code $\mathcal{C}$ is called a \emph{permutation code} if $\mathcal{C}\subseteq S_n$.
It is said to have minimum distance at least $d$ if any two codewords in $\mathcal{C}$ have the Hamming distance at least $d$.

Permutation codes have been extensively studied, see for example \cite{Blake74, BCD79, DF77, DV65, Sle65}. It also has various applications including data transmission over power lines \cite{CCD04, CKL04, FV00, PVY03, Vinck05}, and design of block ciphers \cite{CLT00}. From an extremal perspective, the most natural question for permutation codes is that for given $n$ and $d$, what is the largest possible size $A^{\rm per}(n,d)$ of a length-$n$ permutation code with minimum distance at least $d$?  Let $\vol^{\rm per}(n,d)$ be the volume of a radius-$d$ Hamming ball in $S_n$. 
Once again, the sphere-covering bound \eqref{eq:covering bound} yields
\[A^{\rm per}(n,d+1) \geq \frac{n!}{\vol^{\rm per}(n,d)}.\]
Tait-Vardy-Verstra\"{e}te \cite{TVV13}, Yang-Chen-Yuan \cite{YCY02} and Wang-Zhang-Yang-Ge \cite{WZYG17} further improved this to
\[A^{\rm per}(n,d+1) \geq \Omega(n) \cdot \frac{n!}{\vol^{\rm per}(n,d)} \quad \text{for} \enskip \Omega(n) \le d <n/2.\]
We prove the following strengthening which recovers this bound for a larger range of distance $d$.

\begin{theorem}\label{thm:permutation}
For given $\eps \in (0,1/2)$, there exists a positive constant $c=c_{\eps}$ such that the following holds.
For $\eps n<d<(1-\eps)n$, $A^{\rm per}(n,d+1) \geq cn \cdot \frac{n!}{\vol^{\rm per}(n,d)}$. 
Furthermore, the number of length-$n$ permutation codes with minimum distance at least $d$ is at least
\[
\exp\left(cn^2 \cdot \frac{n!}{\vol^{\rm per}(n,d)}\right).
\]
\end{theorem}


\subsubsection*{Spherical codes}

A \emph{spherical code} of angle $\theta$ in dimension $n$ is a collection of vectors $x_1,\ldots,x_k$ in the unit sphere $\mathbb{S}^{n-1}$
such that $\langle x_i,x_j\rangle \le \cos \theta$ for every $i\ne j$, that is, any two distinct vectors form an angle at least $\theta$. Let $A(n,\theta)$ be the size of the largest spherical code of angle $\theta$ in dimension $n$. 

For $\theta\ge \pi/2$, Rankin \cite{Ran55} determined $A(n,\theta)$ exactly, so from now on we will assume that $\theta \in (0,\pi/2)$. For $x \in \mathbb{S}^{n-1}$, we write 
\[
C_{\theta}(x)=\{y\in \mathbb{S}^{n-1}\colon \langle x,y\rangle \ge \cos \theta\}
\]
for the spherical cap of angular radius $\theta$ around $x$, and let $s_n(\theta)$ denote the normalized surface area of $C_{\theta}(x)$. 

The sphere-covering bound \eqref{eq:covering bound} (observed by Chabauty \cite{Cha53}, Shannon \cite{Sha59}, and Wyner \cite{Wyn65}) implies
\[
A(n,\theta) \ge \frac{1}{s_n(\theta)}=(1+o(1))\sqrt{2\pi n}\cdot \frac{\cos \theta}{\sin^{n-1}\theta}.
\]
For over six decades there have been no improvements to this easy lower bound.
By estimating the expected size of a random spherical code drawn from a Gibbs point process, Jenssen, Joos and Perkins \cite{JJP18} recently improved the lower bound by a linear factor in dimension. 
\begin{theorem}[Jenssen, Joos and Perkins]
\label{thm:sph-code}
    For $\theta\in(0,\pi/2)$, let $c_{\theta}=\log \frac{\sin^2 \theta}{\sqrt{(1-\cos\theta)^2(1+2\cos\theta)}}$. Then,
    $$A(n,\theta)\ge (1+o(1))c_{\theta}\cdot \frac{n}{s_n(\theta)}.$$
\end{theorem}

This bound was very recently further improved by Gil Fern\'andez, Kim, Liu and Pikhurko~\cite{GKLP21+}.

\begin{theorem}[Gil Fern\'andez, Kim, Liu and Pikhurko]\label{thm:sph-code2}
	Let $\theta\in(0,\pi/2)$ be fixed. Then,
	$$A(n,\theta)\geq (1+o(1))\log\frac{\sin\theta}{\sqrt{2}\sin\frac{\theta}{2}}\cdot \frac{n}{s_n(\theta)}, \qquad \mbox{as } n\to\infty.$$
\end{theorem}

 Although \cref{thm: existence of code} is not directly applicable to the continuous setting of spherical codes, we use discretization and the graph theoretic idea in \cref{thm: existence of code} to give a short proof of the improvement of Jenssen, Joos and Perkins \cite{JJP18} in \cref{thm:sph-code}. This answers another question of Jiang and Vardy~\cite{JV04}, who asked whether discretization approach would work for spherical codes. A closely related topic in continuous setting is the sphere packing problem, where a similar approach using integer lattice instead was utilized by Krivelevich, Litsyn and Vardy \cite{KLV04}.
 
 We remark that the best lower bound by Gil Fern\'andez, Kim, Liu and Pikhurko \cite{GKLP21+} in \cref{thm:sph-code2}, however, seems not attainable via discretization and requires to work directly with intrinsic properties of spherical geometry.

\medskip

\noindent\textbf{Organization.} 
The rest of the paper is organized as follows.  
In~\cref{sec:concentration}, we prove \cref{thm:suff-cond-small-int} and
concentration inequalities for Lipschitz functions over slices of Hamming spaces and symmetric group, see~\cref{lem:slice-concentration-q-ary,lem:slice-concentration-perm}.
We then use these concentration inequalities in \cref{sec:small volume} to deduce bounds on the volume of intersections of Hamming/Johnson/permutation balls, see~\cref{lem:volume cube,lem:volume constant-weight,lem:volume permutation}. \cref{sec:prelim} comtains some graph theoretic tools, which will be used in~\cref{sec:GV-bound} to prove~\cref{thm:sph-code,thm:q-ary,thm:constant-weight,thm:permutation} on improvements on sphere-covering bounds. The proof of~\cref{thm:list size} is given in~\cref{sec:list-decoding}.


\section{Proof of~\cref{thm:suff-cond-small-int} and concentration on the slice}
\label{sec:concentration}

In this section we will prove \cref{thm:suff-cond-small-int}, and establish some new concentration inequalities that will be used to verify~(A3) when applying \cref{thm:suff-cond-small-int}. Concentration inequalities are fundamental tools in probabilistic combinatorics and theoretical computer science for proving that nice random variables are near their means. The main principle is that a random function that smoothly depends on many independent random variables should be sharply concentrated. The new concentration inequalities we need are for functions of {\em dependent} random variables. Our proofs use coupling techniques.

\begin{proof}[Proof of~\cref{thm:suff-cond-small-int}]
 Let $T=B(a,r)\cap B(b,r)$, and let $\bm{\eta}\sim B(a,r)$. Then
   $$\frac{\vol(B(a,r)\cap B(b,r))}{\vol(B(a,r))}=\PP(\bm{\eta}\in T).$$
   By definition, $\bm{\eta}$ lies in $T$ if and only if it is of distance at most $r$ from $b$, i.e.
   $$\PP(\bm{\eta}\in T)=\PP(\mathsf{d}(\bm{\eta},b)\le r).$$
   As the metric space has exponential growth at radius $r$, $\PP(\mathsf{d}(\bm{\eta},a)\le r-\alpha k)\le 2e^{-\Omega(k)}$. Thus, 
   \begin{align*}
       \PP(\bm{\eta}\in T)&\le \PP(\bm{\eta}\in T\big| \mathsf{d}(\bm{\eta},a)>r-\alpha k)\cdot \PP(\mathsf{d}(\bm{\eta},a)>r-\alpha k)+\PP(\mathsf{d}(\bm{\eta},a)\le r-\alpha k)\\
       &\le \sum_{i=0}^{\alpha k}\PP\big(\mathsf{d}(\bm{\eta},b)\le r\big| \mathsf{d}(\bm{\eta},a)= r-i\big)\cdot \PP(\mathsf{d}(\bm{\eta},a)= r-i)+2e^{-\Omega(k)}\\
       &\le \max_{0\le i\le \alpha k}\PP\big(\mathsf{d}(\bm{\eta},b)\le r\big| \mathsf{d}(\bm{\eta},a)= r-i\big)+2e^{-\Omega(k)}.
   \end{align*}
   
   Fix an arbitrary $0\le i \le\alpha$, and let
   $\x\sim S(a,r-i)$. Note that, conditioning on $\mathsf{d}(\bm{\eta},a)=r-i$, $\bm{\eta}$ and $\x$ are identically distributed. We thus have
   \begin{align*}
       \PP\big(\mathsf{d}(\bm{\eta},b)\le r\big| \mathsf{d}(\bm{\eta},a)= r-i\big)&=\PP\big(\mathsf{d}(\bm{\eta},b)-\mathsf{d}(\bm{\eta},a)\le i\big| \mathsf{d}(\bm{\eta},a)= r-i\big)\\
       &=\PP(\mathsf{d}(\bm{\x},b)-\mathsf{d}(\x,a)\le i)\\
       &=\PP(\ell_{a,b}(\x)\le i).
   \end{align*}

   Using that $(X,\mathsf{d})$ is $(r,k)$-dispersed with constant $\alpha$, we see that $\E \ell_{a,b}(\x)\ge 2\alpha k$. Consequently, $i-\E \ell_{a,b}(\x)\le i - 2\alpha k \le -\alpha k$. Thus, since $\ell_{a,b}(\x)-\E \ell_{a,b}(\x)$ is $K$-subgaussian, we get
   \begin{align*}
       \PP(\ell_{a,b}(\x)\le i)&=\PP(\ell_{a,b}(\x)-\E \ell_{a,b}(\x)\le i-\E \ell_{a,b}(\x))\\
       &\le \PP(\ell_{a,b}(\x)-\E \ell_{a,b}(\x)\le -\alpha k)\\
       &\le 2e^{-\Omega(k^2/K)},
   \end{align*}
   as desired.
\end{proof}


\subsection{Slices of the $q$-ary cube}

One of the most natural and easy-to-verify smoothness assumptions that one may consider is the so-called bounded differences condition.

\begin{definition}[Bounded differences condition]
 A function
 $f\colon \Omega^n \rightarrow \R$ is said to satisfy the \emph{bounded differences condition} with parameters $(c_1,\ldots,c_n) \in \R^n$ if for every 
 $x,x'\in \Omega^n$
\[
|f(x)-f(x')| \le \sum_{i=1}^n c_i \mathbbm{1}_{\{x_i\ne x_i'\}}.
\]
\end{definition}

In the proof of \cref{thm:list size,thm:q-ary,thm:constant-weight} we will use the following “non-uniform” concentration inequality.

\begin{lemma}\label{lem:slice-concentration-q-ary}
Suppose $f\colon \{0,1,\ldots,q-1\}^n\rightarrow \mathbb{R}$ satisfies the bounded differences condition with parameters $(c_1,\ldots,c_n)$ and that
$\bm{\eta}$ is drawn uniformly at random from $\{0,1,\ldots,q-1\}^n$ subject to $\wt(\bm{\eta})=k$. Then
\[
\PP(|f(\bm{\eta})-\E f(\bm{\eta})|\ge t) \le 2\exp\left(-\frac{t^2}{68\sum_{i=1}^{n}c_i^2}\right) \quad \text{for all $t\ge 0$}.
\]
\end{lemma}

The binary case above is Lemma 2.1 from \cite{KST19}. For completeness, we include its short proof.

\begin{lemma}[\cite{KST19}]\label{lem:concentration-slice}
Suppose $g\colon \{0,1\}^n \rightarrow \mathbb{R}$ satisfies the bounded differences condition with parameters $(c_1,\ldots,c_n)$ and that $\bm{\xi}\in \{0,1\}^n$ is a random vector uniformly distributed in $\binom{[n]}{k}$. Then
\[
\PP(|g(\bm{\xi})-\E g(\bm{\xi})|\ge t) \le 2\exp\left(-\frac{t^2}{8\sum_{i=1}^{n}c_i^2}\right) \quad \text{for all $t\ge 0$}.
\]
\end{lemma}
\begin{proof}
We may assume without loss of generality that $c_{1}\ge\dots\ge c_{n}$.
Consider the Doob martingale $\bm{Z}_{i}=\E\left[g\left(\bm{\xi}\right)\middle|\xi_{1},\dots\xi_{i}\right]$,
so $\bm{Z}_{0}=\E g\left(\bm{\xi}\right)$ and $\bm{Z}_{n}=\bm{Z}_{n-1}=g\left(\bm{\xi}\right)$.
Let $\cL\left(x_{1},\dots,x_{i}\right)$ be the conditional distribution
of $\bm{\xi}$ given $\xi_{1}=x_{1},\dots,\xi_{i}=x_{i}$.

We want to show that
\[
\left|\E\left[g\left(\cL\left(x_{1},\dots,x_{i-1},0\right)\right)\right]-\E\left[g\left(\cL\left(x_{1},\dots,x_{i-1},1\right)\right)\right]\right|\le2c_{i}
\]
for all feasible $x_{1},\dots,x_{i-1}\in\left\{ 0,1\right\} $; this
will imply that $\left|\bm{Z}_{i}-\bm{Z}_{i-1}\right|$ is uniformly bounded
by $2c_{i}$, so the desired result will follow from the Azuma--Hoeffding
bound (see for example \cite[Theorem~22.16]{FK15}).

If $\bm{\xi}$ is distributed as $\cL\left(x_{1},\dots,x_{i-1},0\right)$,
we can change $\xi_{i}$ to 1 and then randomly choose one of the
ones among $\xi_{i+1},\dots,\xi_{n}$ and change it to 0; we thereby
obtain the distribution $\cL\left(x_{1},\dots,x_{i-1},1\right)$. This
provides a coupling between $\cL\left(x_{1},\dots,x_{i-1},0\right)$
and $\cL\left(x_{1},\dots,x_{i-1},1\right)$ that differs in only two
coordinates $i$ and $j>i$, and since $c_{j}\le c_{i}$ this implies
the required bound.
\end{proof}


We also require some standard facts about subgaussian random variables (see for instance \cite[Proposition 2.5.2]{Ver18}).

\begin{lemma}[Subgaussian properties]\label{lem: subgauss-equiv}
Let $\bm{X}$ be a random variable with mean zero. Then the following properties are equivalent.
\begin{itemize}
\item[(i)] There exists $K_1>0$ such that the tails of $\bm{X}$ satisfy
    \[
    \PP(|\bm{X}|\ge t) \le 2\exp\left(-t^2/K_1\right) \quad \text{for all} \enskip t\ge 0.
    \]
\item[(ii)] There exists $K_2>0$ such that the moment generating function of $\bm{X}$ satisfies
\[
\E\exp(\lambda \bm{X}) \le \exp\left(K_2\lambda^2\right) \quad \text{for all} \enskip \lambda \ge 0.
\]
\end{itemize}
In particular, for $(i)\implies (ii)$, we can take $K_2=2K_1$ and for $(ii)\implies (i)$, we can take $K_1=4K_2$. 
\end{lemma}

We now have all the tools to prove \cref{lem:slice-concentration-q-ary}.

\begin{proof}[Proof of \cref{lem:slice-concentration-q-ary}]
Let $\bm{\xi}\in \{0,1\}^n$ be a random vector uniformly distributed in $\binom{[n]}{k}$. Let $\bm{u}$ be drawn uniformly from $[q-1]^n$, independently from $\bm{\xi}$. Then the distribution of $\bm{\eta}$ coincides with the distribution of
\[
\bm{u}\star \bm{\xi}:=(u_1\xi_1,\ldots,u_n\xi_n). 
\]
Writing $\|c\|^2=\sum_{i=1}^nc_i^2$, by \cref{lem: subgauss-equiv}, it suffices to show that
\begin{equation}\label{eq:mangolassi}
    \E_{\bm{u}}\E_{\bm{\xi}}e^{\lambda(f(\bm{u}\star \bm{\xi})-\E_{\bm{u},\bm{\xi}}f(\bm{u}\star \bm{\xi}))}\le e^{17\| c \|^2 \lambda^2}.
\end{equation}

Fix an instance of $\bm{u}$. Note that, as $f(\cdot)$, $f(\bm{u}\star\cdot)$ also satisfies the bounded differences condition with parameters $c=(c_1,\ldots,c_n)$. Then, by \cref{lem:concentration-slice} with $f(\bm{u}\star \cdot)$ playing the role of $g(\cdot)$ and \cref{lem: subgauss-equiv}, we get that
\[
\E_{\bm{\xi}}e^{\lambda(f(\bm{u}\star \bm{\xi})-\E_{\bm{\xi}}f(\bm{u}\star \bm{\xi}))}\le e^{16\|c\|^2\lambda^2}.
\]
Thus,
\begin{align}
    \E_{\bm{u}}\E_{\bm{\xi}}e^{\lambda(f(\bm{u}\star \bm{\xi})-\E_{\bm{u},\bm{\xi}}f(\bm{u}\star \bm{\xi}))}&= e^{-\lambda \E_{\bm{u},\bm{\xi}}f(\bm{u}\star \bm{\xi})}\cdot \E_{\bm{u}}e^{\lambda \E_{\bm{\xi}}f(\bm{u}\star \bm{\xi})}\E_{\bm{\xi}}e^{\lambda (f(\bm{u}\star \bm{\xi})-\E_{\bm{\xi}}f(\bm{u}\star \bm{\xi}))}\nonumber\\
    &\le e^{16\|c\|^2\lambda^2}\cdot \E_{\bm{u}}e^{\lambda (\E_{\bm{\xi}}f(\bm{u}\star \bm{\xi})- \E_{\bm{u}}\E_{\bm{\xi}}f(\bm{u}\star \bm{\xi}) )}.\label{newyearcoming}
\end{align}
It is easy to check that $g(\cdot):=\E_{\bm{\xi}}f(\cdot\star\bm{\xi})$ also has the bounded differences condition with parameters $c$. Thus by McDiarmid's inequality (see for example \cite[Theorem~22.17]{FK15}), $$\PP(|g(\bm{u})-\E_{\bm{u}}g(\bm{u})|\ge t)\le 2e^{-\frac{2t^2}{\|c\|^2}}$$ 
and so by Lemma~\ref{lem: subgauss-equiv}, $$\E_{\bm{u}}e^{\lambda(g(\bm{u})-\E_{\bm{u}}g(\bm{u}))}\le e^{\|c\|^2\lambda^2}.$$
This, together with~\eqref{newyearcoming}, implies~\eqref{eq:mangolassi} and completes the proof.
\end{proof}


\subsection{Slices of the symmetric group}

The proof of \cref{thm:permutation} relies on the following concentration inequality for functions over slices of the symmetric group. We define the \emph{weight} of a permutation $\sigma$ in $S_n$ to be the Hamming distance between $\sigma$ and the identity.

\begin{lemma}\label{lem:slice-concentration-perm}
Let $S_{n,k}$ be the set of all permutations in $S_n$ with weight $k$. Suppose $f\colon S_{n,k}\rightarrow \mathbb{R}$
satisfies
\begin{equation}\label{eq:Lipschitz-permutation}
|f(\sigma)-f(\tau)| \le \Delta(\sigma,\tau) \quad \text{for all $\sigma,\tau \in S_{n,k}$}.
\end{equation}
Let $\bm{\sigma}$ be drawn uniformly at random from $S_{n,k}$. Then
\[
\PP(|f(\bm{\sigma})- \E f(\bm{\sigma})| \ge t) \le 2\exp(-t^2/72k) \quad \text{for all $t\ge 0$}.
\]
\end{lemma}

To prove \cref{lem:slice-concentration-perm} we will use a coupling argument together with two well-known concentration inequalities. The first is a simple consequence of the Azuma--Hoeffding bound, obtained by Wormald \cite[Theorem 2.19]{Wormald}. 

\begin{theorem}[Wormald \cite{Wormald}]
\label{thm:derangement-concent}
Let $D_n \subset S_n$ be the set of derangements, that is, $\sigma \in D_n$ if and only if $\sigma(i) \ne i$ for all $i\in [n]$.
Suppose $f\colon D_n\rightarrow \mathbb{R}$ satisfies 
\[
|f(\sigma)-f(\tau)| \le \Delta(\sigma,\tau) \quad \text{for all $\sigma,\tau \in D_n$}.
\]
Let $\bm{\sigma}$ be a uniformly random element of $D_n$. Then
\[
\PP(|f(\bm{\sigma})-\E f(\bm{\sigma})|\ge t) \le 2\exp(-t^2/4n) \quad \text{for all $t\ge 0$}.
\]
\end{theorem}

We also need a Gaussian bound for Lipshitz functions on slices of the discrete cube, due to Bobkov \cite[Theorem 2.1]{Bobkov}.

\begin{theorem}[Bobkov \cite{Bobkov}]
\label{thm:uniform concentration-slice}
Let $g\colon \binom{[n]}{k} \rightarrow \R$ be a function such that
\[
|g(J)-g(J')| \le 1.
\]
for any $J,J'\in \binom{[n]}{k}$ with $|J\cap J'|=k-1$.
Let $\bm{I}$ be a uniformly random element of $\binom{[n]}{k}$. Then
\[
\PP\left(|g(\bm{I})-\E g(\bm{I})| \ge t\right) \le 2 \exp\left(-\frac{t^2}{\min\{k,n-k\}}\right) \quad \text{for all $t\ge 0$}.
\]
\end{theorem}

\begin{proof}[Proof of \cref{lem:slice-concentration-perm}]
For $J\subseteq [n]$, let $D_J$ be the set of all permutations which has exactly $[n]\setminus J$ as the set of fixed points, in other words,
\begin{equation}\label{eq:derang}
    D_J  = \{ \sigma \in S_n\colon \sigma(i) \neq  i \enskip \text{if and only if} \enskip i\in J \}.
\end{equation}
It is not difficult to see that the following two-step random process 
yields the uniform distribution on the set of all permutations in $S_n$ with weight $k$:
\begin{enumerate}
    \item[1.] Choose a set $\bm{I}$ uniformly at random from $\binom{[n]}{k}$, and
    \item[2.] Choose a permutation $\bm{\sigma}$ uniformly at random from $D_{\bm{I}}$.
\end{enumerate}
For a set $J\in \binom{[n]}{k}$, let $g(J)$ be the average of $f$ over $D_J$, that is,
\[g(J) = \mathbb{E}_{\bm{I},\bm{\sigma}}[f(\bm{\sigma}) \mid \bm{I}=J] = \E_{ \bm{\sigma}\sim D_J} f(\bm{\sigma}).\]

\begin{claim}\label{cl: bdd property g}
For any $J,J'\in \binom{[n]}{k}$ with $|J\cap J'|=k-1$, we have
\[
|g(J)-g(J')| \le 3.
\]
\end{claim}

\begin{poc}
Let $j$ be the element in $J\setminus J'$ and $j'$ be the element in $J'\setminus J$.
For each permutation $\sigma \in D_J$, we define a permutation $\widetilde{\sigma} \in D_{J'}$ as follows: 
\begin{itemize}
  \item $\widetilde{\sigma}(i)=i$ for all $i \notin \{j,j',\sigma^{-1}(j)\}$,
  \item $\widetilde{\sigma}(j)=j$,
  \item $\widetilde{\sigma}(j')=\sigma(j)$, and 
  \item $\widetilde{\sigma}(\sigma^{-1}(j))=j'$.
\end{itemize}

Since $\sigma(j)\ne j$ and $\sigma(j')= j'$, we see that $\sigma^{-1}(j) \notin \{j,j'\}$, $\sigma(j)\ne j'$, and $j' \ne \sigma^{-1}(j)$. Thus $\widetilde{\sigma}$ is a permutation in $D_{J'}$.
Also it is easy to see that the map $\sigma\mapsto \widetilde{\sigma}$ is a bijection from $D_J$ to $D_{J'}$.
As $\sigma$ and $\widetilde{\sigma}$ differ only at three places, by the hypothesis we have
$|f(\sigma)-f(\widetilde{\sigma})|\le \Delta(\sigma,\widetilde{\sigma})=3$.
Therefore,
\begin{align*}
    |g(J)- g(J')| = \big|\E_{\bm{\sigma}\sim D_J}[f(\bm{\sigma})- f(\widetilde{\bm{\sigma}})]\big|\le 3,
\end{align*}
as desired.
\end{poc}
 Let $\mu$ be the mean of $f$. 
Then note that 
\[ \mu = \E g(\bm{I}).\]
By the triangle inequality,
\[
\PP(|f(\bm{\sigma})-\mu|\ge t) \le \PP(|g(\bm{I})-\mu|\ge t/2)+ \PP(|f(\bm{\sigma})-g(\bm{I})|\ge t/2).
\]

For the first term, recalling \cref{cl: bdd property g} and applying 
\cref{thm:uniform concentration-slice} to $\frac{1}{3}g$,
we get
\[
    \mathbb{P}\{ |g(\bm{I}) -\mu|  \geq t/2 \} \leq  2\exp(-t^2/36k). 
\]

For the second term, note that $g(\bm{I})=\E_{\bm{\sigma}\sim D_{\bm{I}}}f(\bm{\sigma})$ for each instance of $\bm{I}$. Once $\bm{I}$ is fixed, for $\bm{\sigma}\sim D_{\bm{I}}$, we can view $f(\bm{\sigma})$ as a function from $D_{\bm{I}}$ to $\mathbb{R}$. Then, by~\cref{eq:Lipschitz-permutation}, we can apply \cref{thm:derangement-concent} to $f$ and get
\[
\PP(|f(\bm{\sigma})-g(\bm{I})|\ge t/2) \le 2\exp(-t^2/16k).
\]
Therefore,
\[\PP(|f(\bm{\sigma})-\mu|\ge t) \le 4 \exp(-t^2/36k).\]
As the left side is at most one, we get $\PP(|f(\bm{\sigma})-\mu|\ge t) \le 2 \exp(-t^2/72k)$.
\end{proof}


\section{Small intersection}
\label{sec:small volume}
In this section, we will verify the conditions of \cref{thm:suff-cond-small-int} for Hamming/Johnson/permutation spaces, using the concentration inequalities proved in previous section, to show that the intersection of balls in these spaces has small volume. 

As these metric spaces $(X,\mathsf{d})$ have the property that the balls of the same radius have the same volume independent of the center point, we will use $\vol(r)$ throughout this section to denote the volume of a radius-$r$ ball in $X$.

We start with the Hamming space. We will need the following standard estimate on the volume of a Hamming ball. 

\begin{lemma}\label{lem:Hammingball-volume}
Suppose that $0<p<1-1/q$ and that $1\le \alpha n \le pn$. Then
\begin{equation*}
    \vol_q(n,\alpha n)=\Theta_{p,q}(1)\cdot \frac{q^{h_q(\alpha)n}}{\sqrt{\alpha n}}.  \end{equation*}
\end{lemma}

The Hamming space satisfies the conditions of \cref{thm:suff-cond-small-int} as follows.

\begin{lemma}
\label{lem:volume cube}
Let $0<p< \frac{q-1}{q}$, and let $k$ be any positive integer.
Consider $X=\{0,1,\ldots,q-1\}^n$ endowed with the Hamming distance $\Delta$.
Then $(X,\Delta)$ satisfies the conditions {\rm (A1)}--{\rm (A3)} of \cref{thm:suff-cond-small-int} as follows.
 \begin{itemize}
      \item[\rm (A1)] $(X,\Delta)$ has exponential growth at radius $pn$ with 
      rate $c=\Omega_{p,q}(1)$.
      
      \item[\rm (A2)] $(X,\Delta)$ is $(pn,k)$-dispersed with constant $\alpha= \frac{1}{2}(1- \frac{pq}{q-1})>0$. 
   
      \item[\rm (A3)] For any $a,b\in X$ with $\Delta(a,b)=k$ and any $0\le i\le \alpha k$, $\ell_{a,b}(\x)-\E \ell_{a,b}(\x)$ is $400k$-subgaussian, where $\ell_{a,b}$ is as in~\eqref{defn:f-slice} and $\x$ is drawn uniformly from 
      $S(a,pn-i)$.
  \end{itemize}
Consequently, for every $a,b \in X$,
\begin{equation}\label{eq:constant}
\frac{\vol(B(a,r)\cap B(b,r))}{\vol(B(a,r))}=2e^{-\Omega_{p,q}(1)\cdot \Delta(a,b)}.
\end{equation}
\end{lemma}
\begin{proof}
(A1) Consider $t<pn$.
By the mean value theorem, $h_q(p)-h_q(p-t/n)=h_q'(x)t/n$ for some $x\in (p-t/n,p)$. Together with Lemma~\ref{lem:Hammingball-volume}, this yields
  \[
  \frac{\vol(pn)}{\vol(pn-t)}=\Omega_{p,q}(1)\cdot \frac{\sqrt{pn-t}}{\sqrt{pn}}\cdot q^{(h_q(p)-h_q(p-t/n))n}=\Omega_{p,q}(1)\cdot\frac{\sqrt{pn-t}}{\sqrt{pn}}\cdot q^{h_q'(x)t}.
  \]
As $x\le p<1-1/q$, we have $h'_q(x)=\log_q(q-1)-\log_q\frac{x}{1-x}\ge \log_q(q-1)-\log_q\frac{p}{1-p}>0$. Letting $\eps=\log_q(q-1)-\log_q\frac{p}{1-p}$, we thus get
  \[
  \frac{\vol(pn)}{\vol(pn-t)} \ge \Omega_{p,q}(1)\cdot\frac{\sqrt{pn-t}}{\sqrt{pn}}\cdot q^{\eps t}.
  \]
  If $t\le pn/2$, then $\frac{\sqrt{pn-t}}{\sqrt{pn}} \ge 1/2$; while $\frac{\sqrt{pn-t}}{\sqrt{pn}}\cdot q^{\eps t}\ge q^{\eps t/2}$ if $pn/2 \le t \leq pn-1$ and $pn$ is sufficiently large. Hence $\frac{\vol(pn-t)}{\vol(pn)} \leq O_{p,q}(1) \cdot q^{-\eps t/2}$ in either case. As the left side is at most one, we conclude that there exists $c=\Omega_{p,q}(1)$ such that $\frac{\vol(pn-t)}{\vol(pn)} \le 2e^{-ct}$ for all $t< pn$.
  
(A2) Consider any two points $a, b\in X$ with $\Delta(a,b)=k$. Let $0\le i \le \alpha k$, and let $\x \sim S(a,pn-i)$. We can assume $a=0^n$ and $b=1^k0^{n-k}$.
Write $\gamma=\PP(x_1\neq 1)=\cdots=\PP(x_1\neq q-1)$ and $\delta=\PP(x_1\neq 0)$. Then $\delta=\frac{pn-i}{n}\le p$. Moreover, note that $(q-1)\gamma+\delta=q-1$, and so $\gamma=1-\frac{\delta}{q-1}\ge 1-\frac{p}{q-1}$. By the linearity of expectation we have
\begin{align*}
   \E \ell_{a,b}({\x})&= \sum_{i=1}^{k}\left(\PP(x_i\neq 1)-\PP(x_i\neq 0)\right)\\
   &=k(\gamma-\delta)\\
   &\ge k\big(1-\frac{p}{q-1}-p \big)= 2\alpha k,
\end{align*}
where the second equality follows from the symmetry.   
   
   (A3) Assume $a=0^{n}$ and $b = 1^k 0^{n-k}$. 
   It is easy to see that the function $\ell_{a,b}$ satisfies the bounded difference condition with parameters $(2,\dots, 2,0,\dots,0)$ where only the first $k$ coordinates are non-zero. Let $0\leq i\leq \alpha k$, and let $\x \sim S(a,pn-i)$.
   By Lemma~\ref{lem:slice-concentration-q-ary}, $\ell_{a,b}(\x)-\mathbb{E}\ell_{a,b}(\x)$ is $400k$-subgaussian.
\end{proof}


Our next result justifies the conditions of \cref{thm:suff-cond-small-int} for the Johnson space.

\begin{lemma}
\label{lem:volume constant-weight}
Let $\beta,\lambda$ and $\eps$ be real numbers with $0<\eps<1/10$ and $0<\beta < (1-\eps)\lambda(1-\lambda)$. Let $k$ be any positive integer.
Consider the slice $X=\binom{[n]}{\lambda n}$ endowed with the Johnson distance $\mathsf{d}$.
Then $(X,\mathsf{d})$ satisfies the conditions {\rm (A1)}--{\rm (A3)} of \cref{thm:suff-cond-small-int} as follows.
 \begin{itemize}
      \item[\rm (A1)] $(X,\mathsf{d})$ has exponential growth at radius $\beta n$ with 
      rate $\eps^2$; 
      
      \item[\rm (A2)] $(X,\mathsf{d})$ is $(\beta n,k)$-dispersed with constant $\eps/2$;
   
      \item[\rm (A3)] For any $a,b\in X$ with $\Delta(a,b)=k$ and any $0\le i\le \eps k$, $\ell_{a,b}(\x)-\E \ell_{a,b}(\x)$ is $8\beta n$-subgaussian, where $\x \sim S(a,\beta n-i)$.
  \end{itemize}
Consequently, for every $a,b \in X$,
\begin{equation}
\frac{\vol(B(a,\beta n)\cap B(b,\beta n))}{\vol(B(a,\beta n))}=2e^{-\Omega_{\eps}(1)\cdot \big(\mathsf{d}(a,b)+\mathsf{d}(a,b)^2/(\beta n)\big)}.
\end{equation}
\end{lemma}
\begin{proof}

(A1)
We wish to show that $\vol(\beta n-t)/\vol(\beta n)\leq 2e^{-\eps^2 t}$ for all $t \le \beta n$. Since the left side is at most one, we can assume $t \ge 1/(2\eps^2)$.
Recall that $\vol(d) = \sum_{i=0}^{d} \binom{\lambda n}{i} \binom{(1-\lambda)n}{i}$ for all non-negative integer $d$.
For $1 \le i\le \beta n$, we have
\begin{align}\label{eq:constant-weight ratio}
\notag
\frac{\binom{\lambda n}{i}\binom{(1-\lambda)n}{i}}{\binom{\lambda n}{i-1}\binom{(1-\lambda)n}{i-1}}&=\frac{(\lambda n-i+1)((1-\lambda)n-i+1)}{i^2}\\ \notag
&\ge \frac{(\lambda-\beta)((1-\lambda)-\beta)}{\beta^2}\\
&=1+\frac{\lambda(1-\lambda)-\beta}{\beta^2}\ge 1+4\eps.
\end{align}
It follows that
\[\vol(\beta n-t) \leq \binom{\lambda n}{\beta n-t} \binom{(1-\lambda)n}{\beta n-t}\cdot\sum_{i=0}^{\beta n-t} (1+4\eps)^{-i} \leq \binom{\lambda n}{\beta n-t} \binom{(1-\lambda)n}{\beta n-t}\cdot\frac{1+4\eps}{4\eps}.\]
Furthermore, \eqref{eq:constant-weight ratio} implies $\vol(\beta n)\ge \binom{\lambda n}{\beta n}\binom{(1-\lambda)n}{\beta n} \ge \binom{\lambda n}{\beta n-t}\binom{(1-\lambda)n}{\beta n-t}\cdot (1+4\eps)^t$. Therefore, we have $\frac{\vol(\beta n-t)}{\vol(\beta n)} \le \frac{1+4\eps}{4\eps}\cdot (1+4\eps)^{-t} \le 2 e^{-\eps^2 t}$ assuming $0<\eps \le 1/10$ and $t\ge 1/(2\eps^2)$.

(A2) Consider any two points $a,b \in X$ with $\mathsf{d}(a,b)=k$. Let $0\le i \le \eps k/2$, and let $\x \sim S(a,\beta n-i)$. We can assume
$a=1^{\lambda n}0^{(1-\lambda)n}$ and $b=0^{k}1^{\lambda n}0^{(1-\lambda)n-k}$.
Since $\x\in \binom{[n]}{\lambda n}$ and $\mathsf{d}(\x,a)=\beta n-i$, we find $\sum_{j=1}^{\lambda n}x_j=(\lambda-\beta)n+i$ and $\sum_{j=\lambda n +1}^{n}x_j=\beta n-i$. We thus get $\E x_1=\cdots=\E x_{\lambda n}=\frac{(\lambda-\beta)n+i}{\lambda n}$ and $\E x_{\lambda n+1}=\cdots=\E x_n=\frac{\beta n-i}{(1-\lambda)n}$, by the symmetry.
Furthermore, notice that 
\[
\ell_{a,b}(\x)=\mathsf{d}(\x,b)-\mathsf{d}(\x,a)=\frac12 \sum_{j=1}^k (2x_j-1)+\frac12 \sum_{j=\lambda n+1}^{\lambda n+k}(1-2x_j)=\sum_{j=1}^{k}x_j-\sum_{j=\lambda n+1}^{\lambda n+k}x_j.
\]
Therefore, by linearity of expectation, we obtain
\[
\E \ell_{a,b}(\x)=k\cdot\left(\frac{(\lambda-\beta)n+i}{\lambda n}-\frac{\beta n-i}{(1-\lambda)n}\right) \ge k\cdot \frac{\lambda(1-\lambda)-\beta}{\lambda(1-\lambda)}\geq \eps k,
\]
as desired.

(A3)
Without loss of generality we can assume $a=1^{\lambda n}0^{(1-\lambda)n}$ and $b=0^{k}1^{\lambda n}0^{(1-\lambda)n-k}$.
We wish to show
$ \PP(|\ell_{a,b}(\x)-\E \ell_{a,b}(\x)|\ge t) \leq 2  e^{-t^2/(8\beta n)}$ for all $t\ge 0$.
As the left side is at most one, we may assume $2  e^{-t^2/(8\beta n)} \leq 1$.
Observe that $\x \sim S(a,\beta n-i)$ is a concatenation of two independent random vectors 
$(x_1,\ldots,x_{\lambda n})\sim \binom{[\lambda n]}{(\lambda-\beta)n+i}$ and $(x_{\lambda n+1},\ldots,x_n)\sim \binom{ [n]\setminus[\lambda n]}{\beta n-i}$. Moreover, we can decompose $\ell_{a,b}(\x)=f(x_1,\ldots,x_{\lambda n})+g(x_{\lambda n+1},\ldots,x_n)$, where $f(x_1,\ldots,x_{\lambda n})=\sum_{j=1}^{k}x_j$ and $g(x_{\lambda n+1},\ldots,x_n)=-\sum_{j=\lambda n+1}^{\lambda n+k}x_j$. Applying
\cref{thm:uniform concentration-slice} to $f$ and $g$, we therefore get
\begin{align*}
      \PP(|\ell_{a,b}(\x)-\E \ell_{a,b}(\x)|\ge t)   &\le \PP( |f-\E f| \ge t/2)+ \PP(|g-\E g| \ge t/2)\\
       & \le 4\exp\Big(-\frac{t^2}{4(\beta n-i)}\Big)  \le 4  e^{-t^2/(4\beta n)} \leq 2  e^{-t^2/(8\beta n)},
   \end{align*}
   where the last inequality holds as $2  e^{-t^2/(8\beta n)}\leq 1$. This completes our proof.
\end{proof}


The last result of this section confirms the conditions of \cref{thm:suff-cond-small-int} for the permutation space.
\begin{lemma}
\label{lem:volume permutation}
Let $0<\eps <0.01$, $1 \le r \leq (1-\eps )n$ and $k \ge 6/\eps$.
Consider the symmetric group $S_n$ endowed with the Hamming distance $\Delta$.
Then $(S_n,\Delta)$ satisfies the conditions {\rm (A1)}--{\rm (A3)} of \cref{thm:suff-cond-small-int} as follows.
 \begin{itemize}
      \item[\rm (A1)] $(S_n,\Delta)$ has exponential growth at radius $r$ with rate $\eps$; 
      
      \item[\rm (A2)] $(S_n,\Delta)$ is $(r,k)$-dispersed with constant $\eps/4$;
   
      \item[\rm (A3)] For any $a,b\in M$ with $\Delta(a,b)=k$ and any $0\le i\le \eps k/4$, $\ell_{a,b}(\x)-\E \ell_{a,b}(\x)$ is $72r$-subgaussian, where $\x \sim S(a,r-i)$.
  \end{itemize}
Consequently, for every $a,b \in S_n$ with $ \Delta(a,b)\ge 6/\eps$,
\begin{equation}
\frac{\vol(B(a,r)\cap B(b,r))}{\vol(B(a,r))} \le 2e^{-\Omega_{\eps}(1)\cdot \left(\Delta(a,b)+\Delta(a,b)^2/r\right)}.
\end{equation}
\end{lemma}
\begin{proof}
(A1) 
We wish to show $\vol(r-t)/\vol(r)\le 2e^{-\eps t}$ for all $t<r$. As the left side is at most one, we may assume $2e^{-\eps t} \le 1$.
It is well known that $\frac{1}{3}|I|!\leq |D_{|I|}|\leq \frac{1}{2}|I|!$ for $|I|\ge 2$ (where $D_I$ is as defined in~\eqref{eq:derang}). Hence,
   \begin{align*}
      \frac{\vol(r-t)}{\vol(r)} &= \frac{ \sum_{I \in \binom{[n]}{\le r-t}} |D_{I}| }{ \sum_{I \in \binom{[n]}{\le r}} |D_{I}| } 
   \leq \frac{1+ \frac{1}{2} \sum_{i=2}^{r-t} \binom{n}{i}i! }{1+ \frac{1}{3} \sum_{i=2}^{r} \binom{n}{i}i! } \\ 
   &\leq \frac{3}{2}\cdot \frac{ 3n(n-1)\cdots (n -r+t+1) }{n(n-1)\cdots (n-r+1)} \leq  \frac{5}{t!} \leq 2 e^{-\eps t},
   \end{align*}
   where the last inequality holds as $\eps \le 0.01$ and $2e^{-\eps t} \le 1$.

(A2) Consider any $a,b \in S_n$ with $\Delta(a,b)=k$. Let $0\le i \le \eps k/4$, and let $\x \sim S(a,r-i)$. We can assume $a\in S_n$
 is the identity permutation and $b$ is a permutation in $D_{[k]}$.
To compute the mean of $\ell_{a,b}(\x)$, we generate $\x$ by first drawing $\bm{I}\sim \binom{[n]}{r-i}$ and then choosing $\x \sim D_{\bm{I}}$. 

Note that for all  $i \in [k]\setminus \bm{I}$ and $j\in \{k+1,\ldots,n\}\setminus \bm{I}$, we have $\x(i) = i \neq b(i)$ and $\x(j)=j=b(j)$.  
Hence, by the linearity of expectation, we have 
\begin{align*}
\E [\Delta(\x,b)\colon \bm{I}=I] &= \sum_{i\in I} \PP[\x(i)\neq b(i)] + |[k]\setminus I| \\
& \geq |I|\left(1 - \frac{ (|I|-1)!}{|D_I|}\right)  + |[k]\setminus I|\\
&\geq (r-i-3)+ |[k]\setminus I|. 
\end{align*}
Here the penultimate inequality holds as there are at most $(|I|-1)!$ permutations fixing one value, and the final inequality follows from the facts that $|D_I|\geq \frac{1}{3}|I|!$ and that $|I|=r-i$. From this we get
\begin{align*}
\E [\Delta(\x,b)]&=\sum_{I\in \binom{[n]}{r-i}} \E [\Delta(\x,b)\colon \bm{I}=I]\cdot \PP(\bm{I}=I)\\
&\ge (r-i-3)+\E |[k]\setminus \bm{I}|\\
&=(r-i-3)+\frac{k(n-r+i)}{n}.
\end{align*}
As $\ell_{a,b}(\x)=\Delta(\x,b)-\Delta(\x,a)=\Delta(\x,b)-(r-i)$, we obtain
\begin{align*}
\E[\ell_{a,b}(\x)]&= \E [\Delta(\x,b)]-(r-i)\\
&\ge \frac{k(n-r+i)}{n}-3\\
&\ge \eps k-3 \ge \eps k/2,
\end{align*}
assuming $r\le (1-\eps)n$ and $k\ge 6/\eps$.

(A3)
For all $x,x' \in B(a,r-i)$, we have
$|\ell_{a,b}(x)-\ell_{a,b}(x')|=|\Delta(x,b)-\Delta(x',b)| \le \Delta(x,x')$.
Hence Lemma~\ref{lem:slice-concentration-perm} implies $\ell_{a,b}(\x)-\E \ell_{a,b}(\x)$ is $72r$-subgaussian.
\end{proof}


\section{Graph theoretic tools}\label{sec:prelim}

We will reduce the lower bound on various codes to lower bound on independence number of some auxiliary graphs. We then show that all the auxiliary graphs are locally sparse. We can then use known bound on independence number of locally sparse graphs. We will use the following variant which is tailored to our needs.

\begin{theorem}\label{thm:locally-sparse}
    	Let $G$ be an $N$-vertex with maximum degree $D$ and minimum degree at least $D/2$. Let $K\in [1,D]$, and let $\Gamma\subseteq G$ be a subgraph induced by the neighborhood of an arbitrary vertex. Suppose there is a partition $V(\Gamma)=B\cup I$ such that 
    	\begin{itemize}
    	    \item every vertex $u\in B$ has degree $\deg_{\Gamma}(u)\le D/K$; and
    	    \item $|I|\le D/K$.
    	\end{itemize}
    	Then the independence number of $G$ is at least $\big(1-o_{K\rightarrow\infty}(1)\big)\frac{N}{D}\log K$, and the number of independent sets in $G$ is at least
	$\exp\left((\frac{1}{8}+o_{K\rightarrow\infty}(1))\frac{N}{D}\log^2 K\right)$. 
\end{theorem}

\begin{remark}
In some of our applications we have $K=D^{\Theta(1)}$, in which case the second conclusion implies that the average size of an independent set in $G$ is at least $\Omega(1)\cdot \frac{N}{D}\log D$. 
\end{remark}

\begin{proof}[Proof of \cref{thm:locally-sparse}]
As $|\Gamma|\le D$, we get 
\[
2e(\Gamma)=\sum_{v\in B}\deg_{\Gamma}(v)+\sum_{v\in I}\deg_{\Gamma}(v)\le |B|\cdot (D/K)+|I|\cdot |\Gamma|\le 2D^2/K.
\]
Hence $\Gamma$ has average degree at most 
$4D/K$. By a result of Hurley and Pirot \cite[Theorem 2]{HP21}, $G$ has chromatic number at most $(1+o_{K\rightarrow\infty}(1))\frac{D}{\log K}$. It follows that the independence number of $G$ is at least 
$\big(1-o_{K\rightarrow\infty}(1)\big)\frac{N}{D}\log K$,
as desired.

For the second statement, we need to introduce some notation. Let $\cI(G)$ be the collection of independent sets of $G$. The {\em hard-core model on $G$ at fugacity} $\lambda>0$ is a probability distribution on $\cI(G)$, where each $I\in \cI(G)$ occurs with probability proportional to $\lambda^{|I|}$. In other words,
\[
\PP[I]=\frac{\lambda^{|I|}}{\sum_{J\in \cI(G)}\lambda^{|J|}}.
\]
The denominator, $P_G(\lambda)=\sum_{J\in \cI(G)}\lambda^{|J|}$, is the {\em partition function} of the hard-core model on $G$. Note that $P_G(\lambda)$ is an increasing function with $P_G(0)=1$ and $P_G(1)=|\cI(G)|$.

The expected size of an independent set drawn from the hard-core model on $G$ at fugacity $\lambda$ is the scaled logarithmic derivative of the partition function:
\begin{equation}\label{eq:occupancy}
\widebar{\alpha}_G(\lambda)=\sum_{I\in \cI(G)}|I|\cdot \PP[I]=\frac{\sum_{I\in \cI(G)}|I|\lambda^{|I|}}{P_G(\lambda)}=\frac{\lambda P'_G(\lambda)}{P_G(\lambda)}=\lambda \cdot (\log P_G(\lambda))'.
\end{equation}

We need a lower bound on $\widebar{\alpha}_G(\lambda)$ for certain range of $\lambda$, due to Davies et al. \cite{DJKP21}. The lower bound is written in terms of the Lambert W function: for $z>0$, $W(z)$ is the unique positive real satisfying $W(z)e^{W(z)}=z$. Note that $W(z)=(1+o(1))\log z$ as $z\rightarrow \infty$.

Consider a graph $G$ that satisfies the assumptions of \cref{thm:locally-sparse}. Let $\lambda_0=\frac{\log K}{D}$ and $\lambda_1=\frac{\sqrt{K}}{D}$. As $e(\Gamma)\le D^2/K$, Theorem 5 in \cite{DJKP21} shows that for all $\lambda \in [\lambda_0,\lambda_1]$ we have
\[
\frac{1}{N}\widebar{\alpha}_G(\lambda) \ge (1+o(1))\frac{\lambda}{1+\lambda}\frac{W(D\log(1+\lambda))}{D\log(1+\lambda)}.
\]
Combining this with \eqref{eq:occupancy} and letting $u_i=W(D\log(1+\lambda_i))$, we find
\begin{align*}
\log P_G(\lambda_1)-\log P_G(\lambda_0) &\ge \frac{N}{D}\int_{\lambda_0}^{\lambda_1}\frac{W(D\log(1+t))}{(1+t)\log(1+t)}\,dt\\
&=\frac{N}{D}\int_{W(D\log(1+\lambda_0))}^{W(D\log(1+\lambda_1))}(1+u)\,du\\
&=\frac{N}{2D} \, [u_1^2+2u_1-u_0^2-2u_0],
\end{align*}
where the first equality follows from change of variable $u=W(D\log(1+t))$. Using the approximations $D\log(1+\lambda_0)=(1+o(1))\log K$, $D\log(1+\lambda_1)=(1+o(1))\sqrt{K}$, and $W(z)=(1+o(1))\log z$, we have $u_0=(1+o(1))\log\log K$ and $u_1=(\frac12+o(1))\log K$. Therefore, we get
\[
\log P_G(\lambda_1)-\log P_G(\lambda_0) \ge \big(\frac{1}{8}+o(1)\big)\frac{N}{D}\log^2 K.
\]
Since $1\le P_G(\lambda_0) \le P_G(\lambda_1) \le |\cI(G)|$, this gives
$\log |\cI(G)| \ge (\frac{1}{8}+o(1))\frac{N}{D}\log^2 K$, as desired. 
\end{proof}


\section{Improvement on Gilbert--Varshamov bounds}\label{sec:GV-bound}

We present in this section a unified short proofs of improvements on  
sphere-covering bounds on various codes by reducing it to lower bound on independence number of an auxiliary graph. In order to use \cref{thm:locally-sparse}, we need to show that the graph is locally sparse. Our strategy is to split the edge count in the subgraph induced by the neighbourhood of a vertex into two parts, one from vertices from the boundary of the Hamming/Johnson/Euclidean ball, and the other from interior vertices of the ball. The contribution from boundary vertices is exponentially small because the volume of the intersection of balls that are far apart is small as we have shown using~\cref{thm:suff-cond-small-int} and concentration of measure.
On the other hand, the contribution from the interior vertices is also small as there are negligible amount of interior vertices using the growth of the balls in such spaces. 

\begin{proof}[Proof of \cref{thm: existence of code}]
Define a graph $G$ whose vertices are points in the metric space $(X,\mathsf{d})$ and two points are adjacent if their distance is at most $r$. It is easy to see that $G$ has $|X|$ vertices, the degree of every vertex is $\vol(r)-1$, and the maximum size of an $(X,\mathsf{d},r)$-code is the independence number $\alpha(G)$ of $G$. Let $\Gamma$ be a subgraph induced by the neighborhood of an arbitrary vertex
$x\in X$. We partition $V(\Gamma)=B\cup I$, where $I$ is the punctured ball of radius $r-t$ centered at $x$. By the assumption, $\frac{|I|}{\vol(r)}=\frac{\vol(r-t)-1}{\vol(r)}\le e^{-K}$.
Consider any vertex $u\in B$. As $r-t < \mathsf{d}(x,u) \le r$, we obtain
\[
\frac{\deg_{\Gamma}(u)}{\vol(r)}=\frac{\vol(B(x,r)\cap B(u,r))}{\vol(r)} \le e^{-K}.
\]
Therefore, \cref{thm: existence of code} is a realization of \cref{thm:locally-sparse}.
\end{proof}

\begin{proof}[Proof of \cref{thm:q-ary,thm:constant-weight,thm:permutation}]
Each of \cref{lem:volume cube,lem:volume constant-weight,lem:volume permutation} verifies the conditions for each of $q$-ary codes, constant-weight codes and permutation codes
for applying \cref{thm: existence of code}, respectively.
Hence, \cref{thm:q-ary,thm:constant-weight,thm:permutation} all follow from \cref{thm: existence of code}.
\end{proof}


\subsection{Spherical codes}
We need two lemmas for the short proof of~\cref{thm:sph-code}. The first one is a folklore result that partitions the sphere into small pieces of equal measure (see e.g.~\cite[Lemma 21]{FS}).
\begin{lemma}\label{lem: partition-sphere}
For each $\delta \in (0,1)$ the sphere $\mathbb{S}^{n-1}$ can be partitioned into $N=(O(1)/\delta)^n$ pieces of equal measure, each of diameter at most $\delta$.
\end{lemma}

The second one is an Euclidean version of results from \cref{sec:small volume}. For a measurable set $A\subset \mathbb{S}^{n-1}$, let $s(A)$ denote the normalized surface area of $A$. Recall that $s_n(\theta)$ is the normalized surface area of a spherical cap of angular radius $\theta$. It is well known that for fixed angle $\theta \in (0,\pi/2)$
\begin{equation}\label{eq:cap}
 s_n(\theta)=\frac{1+o(1)}{\sqrt{2\pi n}}\cdot \frac{\sin^{n-1}\theta}{\cos \theta}.  
\end{equation}

We need a parameter $q_{\theta}$, which is the angular radius of the smallest cap  containing the intersection of two spherical caps of angular radius $\theta$ whose centers are at angle $\theta$. It is straightforward to compute that
\begin{equation}\label{eq:q-theta}
    q_{\theta}=\arcsin\Big(\frac{\sqrt{(\cos\theta-1)^2(1+2\cos\theta)}}{\sin\theta}\Big).
\end{equation}
\begin{lemma}[{\cite[Lemma 6]{JJP18}}]\label{lem:intersection-cap}
   Let $x \in \mathbb{S}^{n-1}$ and $A\subset C_{\theta}(x)$ be measurable with $s(A)>0$. Then
   \[
   \underset{\bm{u}\sim A}{\E}[s(C_{\theta}(\bm{u})\cap A)] \le 2 \cdot s_n(q_{\theta}),
   \]
where $q_{\theta}$ is as in~\eqref{eq:q-theta}.
\end{lemma}

\begin{proof}[Proof of~\cref{thm:sph-code}]
Choose $\delta\ll_{\theta,n} 1$, that is, $\delta$ is less than a suitable function of $\theta$ and $n$. Apply Lemma~\ref{lem: partition-sphere} to partition the unit sphere into $N=(O(1)/\delta)^n$ pieces $P_1,\ldots,P_N$ of equal measure, each with diameter at most $\delta$. For each $i \in [N]$, pick an arbitrary point $v_i$ from $P_i$. Let $G$ be a graph with vertex set being these $N$ chosen points, and two vertices form an edge if the angle between them is less than $\theta$. Then by definition, $A(n,\theta)\ge \alpha(G)$. We first use a packing/covering argument to show that every vertex in $G$ has degree $(1+o(1)) s_{n}(\theta)N$. Write $N[x]:=N(x)\cup\{x\}$ for the closed neighborhood of $x$.

\begin{claim}\label{claim:covering}
For every $x \in V(G)$, 
\[
C_{\theta-2\delta}(x) \subset \bigcup_{\bm{v}_i \in N[x]}P_i \subset C_{\theta+2\delta}(x).
\]
\end{claim}
\begin{poc}
We only prove the first inclusion. Let $y$ be any point in $C_{\theta-2\delta}(x)$, that is, the angle between $y$ and $x$ is at most $\theta-2\delta$. As the $P_i$'s cover the sphere, there exists an index $i$ such that $y \in P_i$. By the assumption on $P_i$, we have $||y-v_i|| \le \delta$. Thus, the angle between $v_i$ and $y$ is $2\arcsin(||y-v_i||/2) \le 2\arcsin(\delta/2) < 2\delta$. It follows from the triangle inequality that the angle between $v_i$ and $x$ is less than $2\delta+(\theta-2\delta)=\theta$, implying $v_i\in N[x]$. Therefore, for every $y\in C_{\theta-2\delta}(x)$ we must have $y \in \bigcup_{v_i \in N[x]}P_i$, as desired.
\end{poc}

Let $x$ be an arbitrary vertex of $G$. Since the $P_i$'s are disjoint subsets of $\mathbb{S}^{n-1}$ of normalized surface area $1/N$, \cref{claim:covering} gives $s_n(\theta-2\delta)N \le |N[x]| \le s_n(\theta+2\delta)N$. Moreover, it follows from \eqref{eq:cap} that $s_n(\theta\pm 2\delta)=(1+O(\delta))^ns_n(\theta)=(1+o(1))s_n(\theta)$. Therefore, every vertex in $G$ has degree $D:=(1+o(1))s_n(\theta)N$. 

Let $K=\frac{s_n(\theta)}{4s_n(q_{\theta})}$. By \eqref{eq:cap}, we obtain $\log K=(1+o(1))\log\frac{\sin\theta}{\sin q_{\theta}}\cdot n=(1+o(1))c_{\theta}\cdot n$.
It suffices to show that we can apply~\cref{thm:locally-sparse} with this choice of $K$. This amounts to proving that for any $x\in V(G)$, the average degree of $G[N(x)]$ is at most $D/K$. For this, we view the average degree of $G[N(x)]$ probabilistically as the expected degree of a uniform random vertex in $N(x)$. 

We partition $N[x]=B\cup I$, where $I=\{v_i\colon P_i\subset C_{\theta}(x)\}$. From \cref{claim:covering}, we know that $\bigcup_{v_i \in B}P_i$ is contained in $C_{\theta+2\delta}(x)\setminus C_{\theta-2\delta}(x)$. Thus, $\delta\ll_{\theta,n} 1$, the number of boundary point is 
\[
|B| \le \big(s_{n}(\theta+2\delta)-s_n(\theta-2\delta)\big)N=O(\delta n)s_{n}(\theta)N=o(D/K),
\]
which is negligible. So it suffices to estimate the average degree of $G[I]$.

Let  $A=\bigcup_{v_i\in I}P_i$, and let $\bm{u}$ be a uniform random point in $A$. Now, as each vertex in $G$ corresponds to a piece of the sphere with the same measure, we can generate $\bm{v}_i\sim I$ by rounding $\bm{u}$ to the vertex $\bm{v_i}$ such that $\bm{u}\in P_i$. Thus, we have by~\cref{lem:intersection-cap} that
$$
\underset{\bm{v}_i\sim I}{\E}[\deg_{G[I]}(\bm{v}_i)]= \underset{\bm{u}\sim A}{\E}[s(C_{\theta}(\bm{u})\cap A)]\cdot N\le 2s_n(q_{\theta})N\le D/K,$$ 
as desired.
\end{proof}


\section{List-decodability of random codes}\label{sec:list-decoding}

In this section, we prove \cref{thm:list size}, which states that a uniformly chosen random code of rate $1-h_q(p)-\eps$ is with high probability {\em not} $(p,(1-o(1))/\eps)$-list decodable. In large part we follow the proof of Guruswami and Narayanan \cite[Theorem 20]{GN14}. As in \cite{GN14} we define a random variable $\bm{W}$ that counts the number of witnesses that certify
the violation of the $(p,L)$-list decodability property. Thus the code is $(p,L)$-list decodable if and only if $\bm{W}=0$. So our job becomes to bound the probability of the event that $\bm{W}=0$. For this we employ the Chebyshev's inequality
\[
\PP(\bm{W}=0) \le \frac{\Var[\bm{W}]}{\E[\bm{W}]^2}.
\]
We then show that $\Var[\bm{W}]/\E[\bm{W}]^2$ is exponentially small, which would finish the proof. To bound the variance, we introduce a new ingredient (\cref{lem:covariance}), whose proof relies crucially on our bound on intersection volume from \cref{lem:volume cube}.

\noindent {\bf Notation.} For the rest of this section, we shall employ the following notation. Given $a \in [q]^n$ and $r\in \mathbb{N}$, we write $\cB_q(a,r)$ for the Hamming ball of radius $r$ centered at $a$. Recall that $\vol_q(n,r)$ is the volume of a radius-$r$ Hamming ball in $[q]^n$, and $\vol_q(n,r;k)$ stands for the volume of the intersections of two radius-$r$ balls whose centers are distance $k$ apart.

\begin{lemma}\label{lem:covariance}
Let $0<p<1-1/q$, $1 \le \ell \le L$ and $\mu:=q^{-n}\vol_q(n,pn)$. There exists a constant $c=c_{p,q}>0$ such that the following holds. Let
\[
\bm{a},\bm{b}, \enskip \x_1,\ldots,\x_{\ell}, \enskip \y_{\ell+1},\ldots,\y_L, \enskip \z_{\ell+1},\ldots,\z_L
\]
be chosen independently and uniformly at random from $[q]^n$. Denote by $\cE_{\ell}$ the event
\[
\Big\{ \x_1,\ldots, \x_{\ell} \in \cB_q(\bm{a},pn)\cap \cB_q(\bm{b},pn), \enskip \y_{\ell+1},\ldots,\y_L\in \cB_q(\bm{a},pn), \enskip \z_{\ell+1},\ldots,\z_L\in \cB_q(\bm{b},pn)\Big\}.
\]
Then
\[
\PP(\cE_{\ell}) \le \min \Big\{\mu^{2L-\ell+1},q^{-n}\mu^{2L-\ell}\left(1+2(q-1)q^{-c\ell}\right)^n\Big\}.
\]
\end{lemma}

\begin{remark}
A version of \cref{lem:covariance}, for the case $q=2$, appeared as \cite[Lemma A.5]{LW21}. The proof of \cite[Lemma A.5]{LW21}, however, does not extend to larger $q$.
\end{remark}

\begin{proof}[Proof of \cref{lem:covariance}]
We first show that the probability of $\cE_{\ell}$ is at most $\mu^{2L-\ell+1}$. For the event $\cE_{\ell}$ to occur, one must have (i) $\bm{a},\bm{b} \in \cB_q(\x_1,pn)$, (ii) $\x_2,\ldots,\x_{\ell},\y_{\ell+1},\ldots,\y_{L}\in \cB_q(\bm{a},pn)$, and (iii) $\z_{\ell+1},\ldots,\z_{L}\in \cB_q(\bm{b},pn)$. Note that the events (i), (ii), (iii) are independent. Conditioned on the position of $\x_1$, (i) occurs with probability $\mu^2$. Given $\bm{a}$ and $\bm{b}$, (ii) and (iii) happen with probability $\mu^{L-1}$ and $\mu^{L-\ell}$, respectively. It follows that $\PP(\cE_{\ell}) \le \mu^{2L-\ell+1}$. 

For the other bound, we first apply the law of total probability to get
\[
\PP(\cE_{\ell})=\sum_{k=0}^n\PP(\Delta(\bm{a},\bm{b})=k)\cdot\PP(\cE_{\ell}\big|\Delta(\bm{a},\bm{b})=k).
\]
Since there are $\binom{n}{k}(q-1)^k$ codewords $b\in [q]^n$ which are at distance $k$ from $a\in [q]^n$, the probability that $\Delta(\bm{a},\bm{b})=k$ is exactly $q^{-n}\binom{n}{k}(q-1)^k$. Conditioned on the positions of $\bm{a}$ and $\bm{b}$ being distance $k$ apart, the probability that $\x_1,\ldots, \x_{\ell} \in \cB_q(\bm{a},pn)\cap \cB_q(\bm{b},pn)$ is $\left(\frac{\vol_q(n,pn;k)}{q^n}\right)^{\ell}=\left(\frac{\vol_q(n,pn;k)}{\vol_q(n,pn)}\right)^{\ell}\mu^{\ell}$. The probability that $\y_{\ell+1},\ldots,\y_L\in \cB_q(\bm{a},pn)$ is $\mu^{L-\ell}$, and the probability that $\z_{\ell+1},\ldots,\z_L\in \cB_q(\bm{a},pn)$ is $\mu^{L-\ell}$. Thus, we have 
\begin{align*}
\PP(\cE_{\ell}\big|\Delta(\bm{a},\bm{b})=k)&=\left(\frac{\vol_q(n,pn;k)}{\vol_q(n,pn)}\right)^{\ell}\mu^{\ell}\cdot \mu^{L-\ell}\cdot \mu^{L-\ell}\\
&=\left(\frac{\vol_q(n,pn;k)}{\vol_q(n,pn)}\right)^{\ell}\mu^{2L-\ell}.
\end{align*}
Therefore, we get the following for some $c=c_{p,q}$ as in \cref{lem:volume cube}.
\begin{align*}
\PP(\cE_{\ell})&=\sum_{k=0}^n q^{-n}\binom{n}{k}(q-1)^k\cdot \left(\frac{\vol_q(n,pn;k)}{\vol_q(n,pn)}\right)^{\ell}\mu^{2L-\ell}\\
(\text{by \cref{lem:volume cube}}) &\le q^{-n}\mu^{2L-\ell}\sum_{k=0}^n\binom{n}{k}(q-1)^k\cdot (2q^{-ck})^{\ell}\\
&=q^{-n}\mu^{2L-\ell}\left(1+2(q-1)q^{-c\ell}\right)^n,
\end{align*}
as desired.
\end{proof}

We are now ready to prove \cref{thm:list size}.

\begin{proof}[Proof of \cref{thm:list size}]
Let $c$ be the positive constant given by \eqref{eq:constant}. Let
\[
\mu:=q^{-n}\vol_q(n,pn), \quad \ell_0:=\frac{1-h_q(p)}{2\eps}, \quad \gamma:=\frac{4(q-1)}{\ln q}\cdot q^{-c\ell_0}, \quad \text{and}\quad L=\frac{1-\gamma}{\eps}.
\]
From \cref{lem:Hammingball-volume}, and recalling that $R=1-h_q(p)-\eps$, we get
\begin{equation}\label{eq:asymptotic-mu}
\mu= \frac{\Theta(1)}{\sqrt{n}}\cdot q^{-(1-h_q(p))n} \quad \text{and} \quad q^{Rn}\mu=\frac{\Theta(1)}{\sqrt{n}}\cdot q^{-\eps n}.    
\end{equation}
Notice that a random $q$-ary code of rate $R$ is simply a random map $\bm{\cC}\colon [q]^{Rn} \rightarrow [q]^n$ where, for each $x \in [q]^{Rn}$, its image $\bm{\cC}(x)$ is chosen independently and uniformly at random from $[q]^n$. For any center $a \in [q]^n$ and any ordered list of $L$ distinct messages $X=(x_1,\ldots,x_L)\in ([q]^{Rn})^L$, we define $\bm{\I}(a,X)$ to be the indicator random variable for the event that $\bm{\cC}(x_1),\ldots, \bm{\cC}(x_L)$ all fall in $\cB_q(a,pn)$, and let $\bm{W}=\sum_{a,X}\bm{\I}(a,X)$. Then $\bm{\cC}$ is $(p,L-1)$-list decodable if and only if $\bm{W}=0$.

We have $\E[\bm{\I}(a,X)]=\PP\big\{\bm{\cC}(x_1),\ldots, \bm{\cC}(x_L)\in  \cB_q(a,pn)\big\}=\mu^L$ and the number of pairs $(a,X)$ is $q^n\cdot \prod\limits_{i=0}^{L-1}(q^{Rn}-i) \ge q^{n}\cdot\frac12 q^{RnL}$. Thus, by linearity of expectation, 
\begin{equation}\label{eq:expectation}
\E[\bm{W}]\ge \frac12 \mu^Lq^{RnL+n}.    
\end{equation}

Observe that if $X$ and $Y$ are two disjoint lists (viewed as sets), then the events $\bm{\I}(a,X)$ and $\bm{\I}(b,Y)$ are independent for any pair of centers $a,b$. It follows that
\begin{align*}
\Var[\bm{W}]&=\sum_{X\cap Y\ne \varnothing}\sum_{a,b} \Big(\E[\bm{\I}(a,X)\bm{\I}(b,Y)]-\E[\bm{\I}(a,X)]\cdot\E[\bm{\I}(b,Y)]\Big)\\
&\le \sum_{X\cap Y\ne \varnothing}\sum_{a,b} \E[\bm{\I}(a,X)\bm{\I}(b,Y)]\\
&=\sum_{\ell=1}^L\sum_{|X\cap Y|=\ell}\sum_{a,b}\PP\big\{\bm{\I}(a,X)=1 \enskip \text{and} \enskip \bm{\I}(b,Y)=1\big\}\\
&=\sum_{\ell=1}^L\sum_{|X\cap Y|=\ell}q^{2n}\cdot \PP_{\bm{a},\bm{b},\bm{\cC}}\big\{\bm{\I}(\bm{a},X)=1 \enskip \text{and} \enskip \bm{\I}(\bm{b},Y)=1\big\},
\end{align*}
where in the last equality we converted the inner summation into an expectation by randomizing over the centers $a$ and $b$. 

Fix a pair $(X,Y)$ with $|X\cap Y|=\ell$, and suppose that the elements of $\bm{\cC}(X)$ are $\x_1,\ldots,\x_{\ell},\y_{\ell+1},\ldots,\y_L$ while the elements of $\bm{\cC}(Y)$ are $\x_1,\ldots,\x_{\ell},\z_{\ell+1},\ldots,\z_L$. Then the event $\big\{\bm{\I}(\bm{a},X)=1 \enskip \text{and} \enskip \bm{\I}(\bm{b},Y)=1\big\}$ is exactly the event $\cE_{\ell}$ in \cref{lem:covariance}. Thus, we can bound the variance of $\bm{W}$ as
\begin{align*}
\Var[\bm{W}] &\le \sum_{\ell=1}^L\sum_{|X\cap Y|=\ell}q^{2n}\cdot \PP(\cE_{\ell})\\
&\le \sum_{\ell=1}^L L^{2L}q^{Rn(2L-\ell)+2n}\cdot \PP(\cE_{\ell}),
\end{align*}
where the second inequality stems from the fact that the number of pairs $(X,Y)$ with $|X\cap Y|=\ell$ is at most $L^{2L}q^{Rn(2L-\ell)}$. We split the summation into $\ell \le \ell_0$ and $\ell>\ell_0$, and get $\Var[\bm{W}]\le V_{\le \ell_0}+V_{>\ell_0}$. From \cref{lem:covariance} and \eqref{eq:expectation}, we find 
\begin{align*}
\frac{V_{\le \ell_0}}{\E[\bm{W}]^2}&\le \frac{4}{\mu^{2L}q^{2RnL+2n}}\sum_{\ell=1}^{\ell_0} L^{2L}q^{Rn(2L-\ell)+2n}\cdot \mu^{2L-\ell+1}\\
&=4L^{2L}\sum_{\ell=1}^{\ell_0}(q^{Rn}\mu)^{-\ell}\cdot \mu\\
(\text{by \eqref{eq:asymptotic-mu}}) &= \Theta(1)\cdot \left(\sqrt{n}q^{\eps n}\right)^{\ell_0} \cdot \frac{\Theta(1)}{\sqrt{n}}q^{-(1-h_q(p))n}\\
(\text{as} \enskip \ell_0 =\frac{1-h_q(p)}{2\eps})&=q^{-\Omega(n)}. 
\end{align*}
Again by appealing to \cref{lem:covariance} and \eqref{eq:expectation}, we see that 
\begin{align*}
\frac{V_{> \ell_0}}{\E[\bm{W}]^2}&\le \frac{4}{\mu^{2L}q^{2RnL+2n}}\sum_{\ell_0<\ell \le L} L^{2L}q^{Rn(2L-\ell)+2n}\cdot q^{-n}\mu^{2L-\ell}\left(1+2(q-1)q^{-c\ell}\right)^n\\
&=4L^{2L}\sum_{\ell_0<\ell \le L}(q^{Rn}\mu)^{-\ell}\cdot \left(\frac{1+2(q-1)q^{-c\ell}}{q}\right)^n\\ 
(\text{by the choice of} \enskip \gamma) & \le 4L^{2L}\sum_{\ell_0<\ell \le L}(q^{Rn}\mu)^{-\ell}\cdot q^{-(1-\gamma/2)n}\\
(\text{by \eqref{eq:asymptotic-mu}}) &=\Theta(1)\cdot \left(\sqrt{n}q^{\eps n}\right)^{L} \cdot q^{-(1-\gamma/2)n}\\
(\text{since} \enskip L= \frac{1-\gamma}{\eps})&=q^{-\Omega(n)}. 
\end{align*}
Putting everything together, we get from Chebyshev's inequality that
\[
\PP(\bm{W}=0) \le \frac{\Var[\bm{W}]}{\E[\bm{W}]^2} \le \frac{V_{\le \ell_0}+V_{>\ell_0}}{\E[\bm{W}]^2} \le  q^{-\Omega(n)}.
\]
Since $\cC$ is $(p,L-1)$-list decodable if and only if $\bm{W}=0$, we conclude that $\bm{\cC}$ is with probability $1-q^{-\Omega(n)}$ not $(p,L-1)$-list decodable.
\end{proof}

\section*{Acknowledgement}
We would like to thank Benny Sudakov for bringing~\cite{KLV04} to our attention.

\end{document}